\documentclass[12pt,reqno]{amsart}
\usepackage[english]{babel}
\usepackage{amssymb}
\usepackage[latin1]{inputenc}

\newtheorem{theorem}{Theorem}[section]
\newtheorem*{theorem*}{Theorem}
\newtheorem{proposition}[theorem]{Proposition}
\newtheorem{lemma}[theorem]{Lemma}
\newtheorem{corollary}[theorem]{Corollary}

\theoremstyle{definition}
\newtheorem{definition}[theorem]{Definition}

\theoremstyle{remark}
\newtheorem{preex}{Example}
\theoremstyle{remark}
\newtheorem{preremark}{Remark}
\theoremstyle{remark}
\newtheorem{preobs}{Observation}

\newenvironment{remark}{\begin{preremark}}{\qed\end{preremark}}

\numberwithin{equation}{section}

\def\R{\mathbb{R}}                              
\def\C{\mathbb{C}}                              
\def\N{\mathbb{N}}

\def\Aut{\rm Aut}

\begin{document}

\baselineskip=17pt

\title{Holomorphic Lie group actions on Danielewski surfaces}

\author{Frank Kutzschebauch}
\curraddr{Institute of Mathematics\\University of Berne\\
Sidlerstrasse 5, CH-3012 Bern\\ Switzerland}\email{Frank.Kutzschebauch@math.unibe.ch}

\author{Andreas Lind}
\curraddr{Department of Mathematics\\ Mid Sweden University\\
SE-851 70 Sundsvall\\ Sweden} \email{Andreas.Lind@miun.se}

\begin{abstract}
We prove that any Lie subgroup $G$ (with finitely many connected components) of an infinite-dimensional topological group $\mathcal G$ which is an amalgamated product of two closed subgroups, can be conjugated to one factor.
We apply this result to classify Lie group actions on Danielewski surfaces by elements of the overshear group (up to conjugation). 
\end{abstract}

\maketitle

\section{Introduction}

The motivation of this paper is the  study of holomorphic automorphisms of Danielewski surfaces. 
These are affine algebraic surfaces defined by 
an equation $D_p := \{xy - p(z) = 0\}$ in $\C^3$, where $p \in \C[z]$ is a polynomial with simple zeros. These surfaces are intensively  studied in affine algebraic geometry, their algebraic automorphism group has been determined by Makar-Limanov \cite{MakarLimanov, MakarLimanov2}. More  results on algebraic automorphisms of Danielewski surfaces can be found in \cite{Crachiola}, \cite{Daigle}, \cite{Daigle2}, \cite{Dubouloz}, \cite{DuboulozPoloni}.

From the holomorphic point of view their
study began in the paper of Kaliman and Kutzschebauch \cite{KutzschebauchKaliman} who proved 
they have the density and volume density property, important features of the so called Anders\'en-Lempert theory.  
For definitions and an  overview over Anders\'en-Lempert theory we refer to \cite{ForstnericKutzschebauch}.

Another important study in the borderland between affine algebraic geometry and complex analysis is the classification
of complete algebraic vector fields on Danielewski surfaces by Leuenberger \cite{Leuenberger}. In fact we explain in Remark \ref{Cstarfibration} how to use his results together with our Classification Theorem \ref{action} to find holomorphic automorphisms 
of Danielewski surfaces which are not contained in the overshear group.

 In~\cite{KutzschebauchLind} we defined the notion of an overshear and shear on Danielewski surfaces as follows. 

\begin{definition}
A mapping $O_{f,g} : D_p \to D_p$ of the form
\[
O_{f,g}(x, y, z)  = \left( x, y + \frac{1}{x}\left(p(z e^{x f(x)} + x g(x)) - p(z)\right), z e^{x f(x)} + x g(x) \right)
\]
(or with the role of 1\textsuperscript{st} and 2\textsuperscript{nd} coordinates exchanged, $IO_{f,g}I$) is called an an \emph{overshear map}, where $f, g : \C \to \C$ are holomorphic functions (and the involution $I $ of $ D_p$  is the map interchanging $x$ and $y$). When $f \equiv 0$, we say that $S_g := O_{0,g}$ is a \emph{shear map} on $D_p$.
\end{definition}

These mappings are automorphisms of $D_p$. The maps of the form $O_{f,g}$
form a group, which we call $O_1$. It can be equivalently described as  the subgroup of $\Aut(D_p)$, leaving the function  $x$ invariant. It is therefore a closed subgroup of $\Aut(D_p)$ (endowed with compact open topology).
Analogously the maps  $I O_{f,g} I$
form a group, the closed subgroup of $\Aut(D_p)$ leaving $y$ invariant, which we call $O_2$.

The main result of \cite{KutzschebauchLind} says that the group generated by overshears, i.e. by $O_1$ and $O_2$, (we call it the {\em overshear group} $\text{OS}(D_p)$) is  dense (w.r.t. the compact-open topology) in the component of the identity of the holomorphic automorphism group $\Aut (D_P)$ of $D_p$. This fact generalizes the classical results of  Anders\'en and Lempert, see~\cite{AndersenLempert}, from $\C^n$. It is worth to be mentioned at this point that $D_p$ for $p$ of degree $1$ is isomorphic to $\C^2$.

In \cite{AndristKutzschebauchLind} the authors together with Andrist proved 
the following  structure result of the overshear group.

\begin{theorem}[Theorem $5.1$ in~\cite{AndristKutzschebauchLind}]\label{StructureTheorem}
Let $D_p$ be a Danielewski surface and assume that $\deg(p) \geq 4$, then the overshear group, $\text{OS}(D_p)$, is a free amalgamated product of $O_1$ and $O_2$.
\end{theorem}

The main result of our paper is the following classification result for Lie group actions on Danielewski
surfaces by elements of the overshear group.

\begin{theorem}\label{action}  Let $D_p$ be a Danielewski surface and assume that $\deg(p) \geq 4$. Let a real connected  Lie group $G$ act on $D_p$ by automorphisms in
$\text{OS} (D_p)$. Then $G$ is abelian, isomorphic to the additive group $(\R^n,+)$ and is conjugated (in $\text{OS} (D_p)$) to a subgroup of $O_1$. 
\end{theorem}
The exact formulas for such actions are described in Corollary \ref{action:formula}.

For the overshear group of $\C^2$ (instead of Danielewski surfaces) many results in the same spirit have been proven by Ahern and Rudin in \cite{AhernRudin} for $G$ a finite cyclic group, by Kutzschebauch and Kraft in~\cite{KutzschebauchKraft} for compact $G$, for one-parameter subgroups in the thesis of Anders\'en  \cite{Andersen2},  by de Fabritiis in \cite{Fabritiis3}, by Ahern, Forstneri\v c and Varolin \cite{AhernForstneric}, \cite{AhernForstnericVarolin}.  For Danielewski surfaces our result is the first of that kind. The proof relies on our second main result, which seems to be of independent interest.

\begin{theorem}\label{Lie}
Let $\mathcal G$ be a topological group which is a free amalgamated product  $O*_{O \cap L} L$ of two  closed subgroups $O, L$. Furthermore let $G$ be a Lie group with finitely many connected components and $\varphi \colon G \to \mathcal G$ be a continuous group homomorphisms. Then $\varphi(G)$ is conjugate to a subgroup of $O$ or $L$.
\end{theorem}

The outline of this paper is the following. In section~\ref{LieSubgroups} we prove Theorem \ref{Lie}. In section \ref{Classification} we prove Theorem
\ref{action}. In section \ref{Examples} we apply Theorem \ref{StructureTheorem} to give new examples of holomorphic automorphisms of $D_p$ not contained in the overshear group $\text{OS} (D_p)$.

\section{Lie subgroups of a free amalgamated product}\label{LieSubgroups}

The aim of this section is to prove the following theorem. For the notion of amalgamated product we refer the reader to \cite{Serre}.

\begin{theorem}\label{Liesubgroup}
Let $\mathcal G$ be a topological group which is a free amalgamated product  $O*_{O \cap L} L$ of two  closed subgroups $O, L$. Furthermore let $G$ be a Lie group with finitely many connected components and $\varphi \colon G \to \mathcal G$ be a continuous group homomorphism . Then $\varphi(G)$ is conjugate to a subgroup of $O$ or $L$.
\end{theorem}

\noindent We need the following facts:

\begin{proposition}\label{ConjugateProp}
Every element of a free amalgamated product $O*_{O \cap L} L$ is conjugate either to an element of $O$ or $L$ or to a cyclically reduced element. Every cyclically reduced element is of infinite order.
\end{proposition}
\begin{proof}
See Proposition~$2$ in section~$1.3$ in~\cite{Serre}.
\end{proof}

\begin{lemma}\label{bounded}
A subgroup $H$ of a free amalgamated product $O*_{O \cap L} L$ is conjugate to a subgroup of $O$ or $L$ if and only if $H$ is of bounded length.
\end{lemma}
\begin{proof}
This is a direct consequence of Proposition~\ref{ConjugateProp}.
\end{proof}

\begin{lemma}\label{CommutingLemma}
Let $g_1$ and $g_2$ be two commuting elements of $O *_{O \cap L} L$ with lengths $\geq 1$, then $l(g_1)$ and $l(g_2)$ are both even or both odd.
\end{lemma}
\begin{proof}
Assume that $g_1=a_1\cdots a_m$ and $g_2 = b_1 \cdots b_n$ are two commuting elements. Assume, for a contradiction, that $l(g_1)$ is even and that $l(g_2)$ is odd. Since $g_1$ has even length, the first and last element of the chain $a_1, \dots, a_m$ have to alter between $O$ and $L$. Similarly, the first and last element of the chain $g_2$s has to be contained in either $O$ or $L$.  

Assume first that $a_1 \in O$ and $a_m \in L$ and that $b_1,b_n \in O$. Then, since $a_m$ and $b_1$ alter between $L$ and $O$, $l(g_1g_2) = m+n$. The assumption that $g_1$ and $g_2$ are commuting, yields that the corresponding length of $g_2 \cdot g_1$ has to be the same as the length of $g_1 \cdot g_2$. Clearly
$$b_1 \cdots b_n \cdot a_1 \cdots a_m = b_1 \cdots b_{n-1} \cdot c \cdot a_2 \cdots a_m,$$
where $c=b_n \cdot a_1 \in O$. Hence $l(g_2g_1) = m + n -1 < m+n = l(g_1g_2)$, which contradicts our assumption. 

If we assume that $a_1 \in O$ and $a_m \in L$ and that $b_1,b_n \in L$ a similar contradiction is obtained. In fact, $l(g_1g_2)=m+n-1 < m+n = l(g_2g_1)$. 

Similar calculations are obtained if $a_1 \in L$ and $a_m \in O$, where we have to consider both of the cases $b_1,b_n \in L$ and $b_1,b_n \in O$.
\end{proof}

\begin{lemma}\label{roots}
If an element $g$ of a free amalgamated product $O*_{O \cap L} L$ has roots of arbitrary order, then it
is conjugate to an element in $O$ or to an element in $L$.
\end{lemma}
\begin{proof}
Assume that $g$ is not conjugate to an element in $O$ or to an element in $L$. Then, by Proposition~\ref{ConjugateProp}, $g$ is conjugate to a cyclically reduced element, say $h^{-1}gh$, which has even length $\geq 2$ by definition of a cyclically reduced element. For each $n > 0$ we have that $h^{-1}gh= h^{-1}(g^{1/n})^nh$, since $g$ as roots of arbitrary order. Hence $h^{-1}g^{1/n}h$ is not an element of $O$ or $L$, since it equals $h^{-1}gh$. Furthermore
\begin{multline*}
h^{-1}(g^{1/n})^nh \cdot h^{-1}gh = h^{-1}(g^{1/n})^ngh = h^{-1} ggh = \\ 
 = h^{-1}g (g^{1/n})^n h =  h^{-1}g h \cdot h^{-1} (g^{1/n})^n h
\end{multline*}

\noindent we conclude that $h^{-1}gh$ and $h^{-1}g^{1/n}h$ commute. Whence, Lemma~\ref{CommutingLemma} implies that $h^{-1}g^{1/n}h$ has even length (since $h^{-1}gh$ has even length), and is thus cyclically reduced. Hence
$$l(h^{-1}gh)=l(h^{-1}(g^{1/n})^nh)= |n|l(h^{-1}g^{1/n}h) \geq |n|\;,$$
for all $n > 0$, contradicting the fact that all elements of $O*_{O \cap L} L$ have finite length.
\end{proof}

First let us establish Theorem~\ref{Liesubgroup} in the case of a one-parameter subgroup:

\begin{proposition}\label{oneparameter}
Let $\mathcal G$ be a topological group which is a free amalgamated product $O*_{O \cap L} L$ of two closed subgroups $O$ and $L$. Let $\varphi \colon \R \to \mathcal G$ be a continuous one-parameter subgroup. Then $\varphi(\R)$ is conjugate to a subgroup of $O$ or $L$.
\end{proposition}
\begin{proof}
Since $\varphi$ is a group homomorphism, we know that $\varphi(1)$ and $\varphi(\sqrt{2})$ have roots of all orders. Hence, we can use Lemma~\ref{roots} to conjugate both elements to $O$ or $L$. Consider the dense subgroup $H = \{ m+n\sqrt{2} : m,n \in \mathbb{Z}\}$ of $\mathbb{R}$. Since
$$l(\varphi(m+n\sqrt{2})) = l(\varphi(m)\varphi(n\sqrt{2})) \leq l(\varphi(1)^m\varphi(\sqrt{2})^n)$$
we conclude that $\varphi(H)$ have bounded length. Therefore, Lemma~\ref{bounded} implies that $\varphi(H)$ is conjugate to $O$ or $L$. Let $c \in O *_{O \cap L} L$ be an element such that $c \varphi(H) c^{-1}$ is contained in $O$ or $L$. Finally, as $O$ and $L$ are closed we get that $$c\varphi(\overline{H})c^{-1} = c\varphi(\mathbb{R})c^{-1} \subseteq \overline{c\varphi(H)c^{-1}}$$ 
is contained in $O$ or $L$.
\end{proof}

The key ingredient in the proof of Theorem~\ref{Lie} will rely on the following result which seems to be of independent interest. In the language 
of \cite{SilvaLeite} this means that every Lie group $G$ is uniformly finitely generated by one-parameter subgroups.

\begin{proposition}\label{main}
For any connected real Lie group $G$ there are finitely many elements $V_i \in \text{Lie}(G)$,  $i=1, 2, \ldots , N$ for which the product map of the one-parameter subgroups
$$\Phi_{V_1, V_2, \dots, V_N} : \R^N \to G$$
defined by
$$(t_1, t_2, \ldots , t_N) \mapsto \exp(t_1 V_1) \exp(t_2 V_2) \cdots \exp(t_N V_N)$$
is surjective.
\end{proposition}
\begin{proof}
By Levi-Malcev decomposition \cite{Levi} and Iwasawa decomposition \cite{Iwasawa} we can write
$$G = S \cdot R = K\cdot A\cdot  N\cdot  R\;,$$
where $S$ is semisimple, $R$ is solvable, $A$ is abelian, $N$ is nilpotent and $K$ is compact. 

If we can prove the claim of the proposition for each of the factors in the above decomposition we will be done.

\noindent For abelian groups the fact holds trivially.

\smallskip
\noindent {\bf Case 1:} $K$ a compact connected Lie group: Take any basis $(k_1,\dots , k_n)$ of the Lie algebra $\text{Lie} (K)$.
Then the product map $\Phi_{k_1, k_2, \ldots, k_n} : \R^n \to K$ is a submersion at the unit element.
Thus its image contains an open neighborhood $U$ of the unit element. Since the powers of a neighborhood
$U$ of the unit element in any connected Lie group cover the whole group, for a compact Lie group $K$ there is a finite number $m$ such that $U^m = K$. This means that for our purpose $\Phi_{k_1, k_2, \ldots, k_n}^m : \R^{nm} \to K$ is surjective.

\smallskip
\noindent {\bf Case 2:} Consider $N$, a nilpotent connected Lie group.
Then $N \cong  \tilde N /\Gamma$ for the universal covering $\tilde N$ and $\Gamma $ a normal discrete subgroup of $\tilde N$. Since the exponential map for  $\tilde N$ factors over $\pi : \tilde N \to N $ it is enough to prove 
the claim for simply connected $N$.

Then, the following fact (due to Malcev~\cite{Malcev}) is true:
If $N$ is simply connected then for a certain (Malcev) basis $(V_1, \dots, V_n)$ of $\text{Lie}(N)$ the map  $(t_1, t_2, \ldots , t_n) \mapsto \exp{t_1 V_1 + t_2 V_2 + \ldots + t_n V_n}$ is a diffeomorphism. 
We will now prove the claim by induction of the length of the lower central series of $\text{Lie} (N)$. For length $1$ the group is abelian and the fact holds trivially. 
Let $g = \exp (t_1 V_1 + t_2 V_2 + \ldots t_n V_n)$.
By  repeated use of Lemma \ref{Campbell}  we write
\begin{multline}
g = \exp (t_1 V_1) \exp (t_2 V_2 + \ldots +t_n V_n) \exp K_1 \\
= \exp (t_1 V_1) \exp (t_2 V_2)  \exp (t_3 V_3 + \ldots +t_n V_n) \exp K_2 \exp K_1 \\
= \exp (t_1 V_1) \exp (t_2 V_2) \ldots \exp (t_n V_n)  \exp K_n \ldots \exp K_2 \exp K_1 
\end{multline}
with  $K_i \in [\text{Lie} (N),\text{Lie} (N)]$.

Since $[\text{Lie} (N),\text{Lie} (N)]$ has shorter length of lower central series, by the  induction hypothesis each of the factors $\exp K_i$ is a product of one parameter subgroups. This proves the claim.

\smallskip
\noindent {\bf Case 3:} $R$ is solvable:  Let $R'$ denote denote the commutator subgroup of $R$. Then $R'$ is
nilpotent and $A:=R/ R'$ is abelian. If $x \in R$ is any element, we can per definition write its image $\bar x$ in $A$
as $\bar x = \exp (t_1 A_1) \cdots \exp (t_n A_n)$ for some $A_i$:s in $\text{Lie}(A)$ which form a basis. Let $\pi : \text{Lie} (R) \to \text{Lie}(A)$ denote the quotient map and let $\tilde A_i \in \text{Lie} (R)$ be elements with $\pi (\tilde A_i) = A_i$. Thus we get $x = \exp (t_1 \tilde A_1) \cdots \exp (t_n \tilde A_n) g$ for some
$g \in R'$. Since $R'$ is nilpotent this reduces our problem to case 2.

\end{proof}
\begin{lemma}\label{Campbell} For a nilpotent Lie group $G$ with Lie algebra 
$\mathfrak{g} =Lie (G)$ and $x, y \in \mathfrak{g}$ there is $K(x,y) \in [\mathfrak{g},\mathfrak{g}] $ with $$\exp (x+y) = \exp (x) \exp (y) \exp (K(x,y))$$

\end{lemma}
\begin{proof}
The key fact is the Baker-Campbell-Hausdorff formula proven by Dynkin in \cite{Dynkin}. In the nilpotent case it says that the is a finite sum of iterated Lie brackets $Z (x,y)$ (number of iterations of brackets bounded by the lower central series of $\mathfrak{g}$) such that for all $x, y\in \mathfrak{g}$
$$\exp (x) \exp (y) = \exp Z(x,y).$$
Moreover $Z(x, y) = x + y + [x, y] + \mathrm{ higher \  brackets}$.
Now 
\begin{multline}
    \exp (x+y) = \exp (x) \exp (y) \exp (-Z(x,y) \exp (x+y)) \\ 
    =\exp (x) \exp (y) \exp (Z(-Z(x,y), x+y))
\end{multline}

Setting $K(x,y) := Z(-Z(x,y), x+y))$ finishes the proof, since
the terms without bracket cancel, i.e., $K(x,y) \in [\mathfrak{g},\mathfrak{g}]$. 
\end{proof}

\noindent Now we are ready to prove the main result of this section.

\begin{proof}[Proof of Theorem~\ref{Lie}]
Let $G_0$ denote a connected component of $G$ containing the identity. By Proposition~\ref{main}, there are finitely many one-parameter subgroups $\R_i$ such that the product map $\R_1 \times  \R_2 \times \cdots \times \R_N \to G_0$ is surjective. By Proposition~\ref{oneparameter} and Lemma~\ref{bounded}, the elements of each of the $\varphi (\R_i)$ have bounded length, say
$a(i)$. Thus the length of the elements in $\varphi (G_0)$ is bounded by $\sum_{i=1}^N a(i)$. As
$G$ has only finitely many connected components the lengths of elements of $\varphi (G)$
are bounded. The assertion now follows from Lemma~\ref{bounded}.
\end{proof}
\section{Classification of Lie group actions by overshears}
\label{Classification}

In this section we prove Theorem \ref{action} from the introduction. We  assume $\text{deg} (p) \ge 4$ and use Theorem \ref{StructureTheorem} from the introduction stating that $\text{OS}(D_p)$ is a free amalgamated product $O_1 * O_2$, where $O_1$ is generated by $O_{f,g}^x$ and $O_2$ is generated by $IO_{f,g}^xI$. By Theorem~\ref{Liesubgroup} we can conjugate any Lie group $G$ with finitely many components acting continuously on $D_p$ by elements of $\text{OS}(D_p)$  into $O_1$ or $O_2$. Without loss of generality
we can assume that we can conjugate any connected Lie subgroup $G$ of $\text{OS}(D_p)$, in particular any one-parameter subgroup, to $O_1$.
Now we have reduced our problem to classify Lie subgroups of $O_1$. We start 
with one-parameter subgroups.


We recall the definitions of overshear fields and shear fields from \cite{KutzschebauchLind}.
\begin{itemize}
\item [$(V1)$] $OF_{f,g}^x :=  p'(z)(zf(x)+g(x))\frac{\partial}{\partial y} + x(zf(x)+g(x))\frac{\partial}{\partial z}$

\item [$(V2)$] $SF_f^x := p^\prime(z)f(x)\frac{\partial}{\partial y} + xf(x)\frac{\partial}{\partial z}$

\end{itemize}
where $f,g$ are entire functions on $\C$. In the special case  $f \equiv 0$ then $OF_{f,g}^x$ is the shear field $SF_g^x$. 

The set of overshear fields is a Lie algebra which consists of complete vector fields only. The formula for the bracket is given by equation~\ref{OvershearGivesShear}.

Any one-parameter subgroup of $\Aut (D_p)$ which is contained in the overshear group $O_1$ is the flow of an overshear field. Let us prove this. The connection between a vector field
$V(x,y,z)$ and the flow $\varphi (x,y,z,t)$ is given by the ODE
\begin{equation}
\frac {d} {dt}\vert_{t=t_0} \varphi (x,y,z,t) = V (\varphi (x,y,z,t_0)), \quad \varphi (x,y,z,0)=(x,y,z))
    \label{flow}
\end{equation} 
Since any action of a real Lie group on a complex space by holomorphic automorphisms is real analytic \cite[1.6]{Akhiezer} we can write the flow
$\varphi (x,y,z,t)= (x, \ldots , z \exp (x f(t,x)) + x g(t,x))$ contained in $O_1$ as 
$$\left( x, \ldots , z \exp( x \sum_{i=0}^\infty  f_i(x) t^i)  + x \sum_{i=0}^\infty g_i (x) t^i \right)$$ for 
entire functions $f_i$ and $g_i$. Using equation~\ref{flow} for $t_0=0$ leads to
$V(x,y,z,t) = p'(z)(z f_0(x)+g_1(x))\frac{\partial}{\partial y}+\{x f_1(x) \exp (x f_0 (x)) z + x g_1 (x) \}\frac{\partial}{\partial z}$, an overshear field.

Calculating the commutator we find that for any $f,g,h$ and $k$, entire functions on $\C$, we have
\begin{equation}\label{OvershearGivesShear}
    [OF_{f,g}^x, OF_{h,k}^x] = x \cdot SF_{gh-kf}\end{equation}
In particular shear fields commute and 
\begin{equation}\label{ShearAndOvershear}
[SF_h^x,OF_{f,g}^x] = x\cdot SF_{fh}^x = xf(x)\cdot SF_h^x.
\end{equation}
\begin{proposition}\label{Infinite}
Let $f,g$ and $h$ be fixed holomorphic functions with $f,h \not\equiv 0$. Then the Lie algebra ${\text Lie}(OF_{f,g}^x,SF_h^x)$ generated by $OF_{f,g}^x$ and $SF_h^x$ is of infinite dimension.
\end{proposition}
\begin{proof}
By expression (\ref{ShearAndOvershear}), and the fact that shear fields commute, we get that
$$\text{Lie}(OF_{f,g}^x,SF_h^x) = \text{span}\{OF_{f,g}^x,SF_{x^nf^nh}; n=0,1,2\dots\}$$
Assume that the Lie algebra is of finite dimension. This means that there is an $n$ and there are constants $a_0,\dots,a_n,b$ such that
$$bOF_{f,g}^x + \sum_{j=0}^n a_jx^jf^j(x)SF_h^x = x^{n+1}f^{n+1}(x)SF_h^x.$$
It follows that $b=0$, whence we get a functional equation of the form
$$\sum_{j=0}^n a_jy^j(x) = y^{n+1}(x)\;,$$
where $y$ is holomorphic and has a zero at $x=0$. This is impossible for non-zero functions $y$, since the right hand side has a higher order of vanishing at $x=0$ than the left hand side.
\end{proof}

\begin{proposition}\label{Abelian}
Let $\mathfrak{g}$ be a Lie algebra contained in $\text{OS}_1$  and suppose that $ {\text dim}(\mathfrak{g}) < +\infty$. Then $\mathfrak{g}$ is abelian.
\end{proposition}

\begin{proof}
Assume that $\mathfrak{g}$ is not abelian. Let $\Theta_1, \Theta_2 \in \mathfrak{g}$ be two non-commuting  vector fields. As explained above they are 
overshear fields and since they  do not commute their bracket $[\Theta_1, \Theta_2]$
is by equation~\ref{OvershearGivesShear} a non-trivial shear field. Now the result follows from Proposition \ref{Infinite}.
\end{proof}

\begin{proof} (of Theorem \ref{action})
As explained in the beginning of the section, the action of $G$ on $D_p$ by overshears  can  be conjugated into $O_1$. The action of $G$ by elements of $O_1$ gives rise to a Lie algebra homomorphism of $\text{Lie} (G)$ into the Lie algebra of vector fields on $D_p$ fixing the variable $x$. This Lie algebra is exactly the set of overshear vector fields $OF^x_{f,g}$ (which consists of complete fields only). By Proposition \ref{Abelian} the finite dimensional Lie algebra  $\text{Lie} (G)$ has to be abelian. 
Since all one-parameter subgroups of $G$ give rise to an overshear vector field, they are isomorphic to $(\R,+)$ (not $S^1$). Thus $G$ is isomorphic to the additive group  $\R^n$ generated by the flows of $n$ linear independent commuting
overshear vector fields $OF^x_{f_i, g_i}, i= 1, 2, \ldots, n$ which commute. By formula \ref{OvershearGivesShear} this is equivalent to $f_i g_j - f_j g_i = 0$ $\forall i, j$. An equivalent  way of expressing this is that the meromorphic functions
$h_i := \frac {g_i }{f_i} $ are the same for all $i$  or that  all $f_i$ are identically zero. 

\end{proof} 

\begin{corollary}\label{action:formula} Suppose $\deg (p) \ge 4$. Every one-parameter subgroup of $\text{OS} (D_p)$ is conjugate by elements of $\text{OS}(D_p)$ to the flow of an overshear field $OF^x_{f,g}$ which in turn is given by the formula
\begin{multline*} 
 (x,y,z,t) \mapsto \\ \left(x, y + \frac{p\left(e^{x f(x) t} z + \frac{g(x)}{f(x)} (e^{x f(x) t} -1)\right)-p(z) } {x}, e^{x f(x) t} z + \frac{g(x)}{f(x)} (e^{x f(x) t} -1) \right).
 \end{multline*}
 Here the expression $\frac{e^{ab}-1}{a}$ for $a=0$ is interpreted as the limit of this expression for $a\to 0$,i .e., as $b$.
\end{corollary}

\begin{remark} It is directly seen from Theorem \ref{action}  that any action of a real  Lie group $G$ on $D_p$  extends to a holomorphic action of  the universal complexification $G^\C$, which in our case has  just the additive group $\C^n$ as connected component. This is a general fact proven by the first author in \cite{Kutzschebauch}. 
\end{remark}

\section{Examples of automorphisms of $D_p$ not contained in $\text{OS} (D_p)$}\label{Examples}

In~\cite{AndristKutzschebauchLind} it is shown that the overshear group is a proper subset of the automorphism group. In fact,  using Nevanlinna theory, there it  is shown that the hyperbolic mapping
$$(x,y,z) \mapsto (xe^z, ye^{-z}, z)$$
is not contained in the overshear group. This is analogous to the result by Anders\'en,~\cite{Andersen}, who showed that the automorphism of  $\C^2$ defined by
$$(x,y) \mapsto (xe^{xy}, ye^{-xy})$$
is not a finite compositions of shears. Hence the shear group is a proper subgroup of the group  of volume-preserving automorphisms. For another proof of this fact see also \cite{KutzschebauchKraft}. 
Note that our Classification Theorem \ref{action}  immediately implies that the elements of the $\C^*$-action $\lambda \mapsto (\lambda x, \lambda^{-1} y, z)$ can not all be contained in $\text{OS} (D_p)$, since there are no $S^1$-actions 
in $\text{OS} (D_p)$.

We will present yet another way of finding an automorphism of a Danielewski surface which is not a composition of overshears.

\begin{theorem}\label{HyperbolicAut}
Assume that $\deg(p) \geq 4$. Then, the overshear group $\text{OS}(D_p)$ is a proper subset of the component of the identity of $\text{Aut}_\text{hol}(D_p)$.
\end{theorem}
\begin{proof}
We look at complete algebraic vector fields on Danielewski surfaces. These are algebraic vectorfields which are globally integrable, however their flow maps are merely holomorphic maps. As shown in \cite{KalimanKutzschebauchLeuenberger}
there is always a $\C$- or a $\C^*$-fibration adapted to these vector fields. That is, there is a map $\pi: D_p \to \C$ such that the flow of the complete field $\theta$ maps fibers of $\pi$ to fibers of $\pi$. These maps $\pi$ have general fibre $\C$
or $\C^*$. In case of at least two exceptional fibers the vector field $\theta$ has to preserve each fibre, i.e., it is tangential to the fibers of $\pi$. For example the overshear fields in $\text{OS}_1$ have adapted fibration $\pi_0 : (x,y,z) \mapsto x$. They are tangential to this  $\C$-fibration, the fibres outside $x=0$ are parametrized  by  $z \in \C$ via $\displaystyle z \mapsto \left(x, \frac {p(z)}{x}, z\right) $. The exceptional fibre is $\pi_0^{-1} (0)$ consisting of $\deg (p)$ copies of $\C$, one for each zero $z_i$ of the polynomial $p$ and parametrized by $y \in \C$ via $y \mapsto (0, y, z_i)$. A typical example of a field with adapted $\C^*$-fibration is the hyperbolic field $x \frac{ \partial} {\partial x} -y \frac{ \partial} {\partial y}$ with adapted fibration $(x,y,z) \mapsto z$. There are $\deg (p)$ exceptional fibers at the zeros of the polynomial $p$, each of them isomorphic to the cross of axis $x y =0$. The same $\C^*$-fibration is adapted to the field $\displaystyle f(z) \left(x \frac{ \partial} {\partial x} -y \frac{ \partial} {\partial y}\right)$ for a nontrivial polynomial $f$. 

Now take any complete algebraic vector  field $\theta$  with an adapted $\C^*$-fibration (and thus generic orbits $\C^*$). Assume that the flow maps (or time-$t$ maps) $\varphi_t \in {\rm Aut}_{hol} (D_p)$  of $\theta$ are all contained in the overshear group $\text{OS} (D_p)$.
Then by Theorem \ref{action} this one-parameter subgroup $t \mapsto \varphi_t$ can be conjugated into $O_1$. This would mean that the one-parameter subgroup would be conjugate to a one-parameter subgroup of an overshear field $OF^x_{f,g}$ (since these are all complete fields respecting the fibration x). This would imply that the generic orbit of the  overshear field is $\C^*$, which is equivalent to $f\ne 0$. However, the  generic orbits of these fields   $OF_{f,g}$ (isomorphic to $\C^*$) are not closed in $D_p$, they contain a fixed point in their closure. Thus our assumption that all $\varphi_t$ are contained in $\text{OS}(D_p)$ leads to a contradiction. In particular we have shown that for any non-zero entire function $f$ there is a $t \in \R$ such that the time $t$-map  of the hyperbolic field given by $$ (x,y,z) \mapsto (xe^{f(z)t},ye^{-f(z)t},z)$$ is not contained in $\text{OS}(D_p)$. 

\end{proof}

\begin{remark} \label{Cstarfibration} More examples of complete algebraic vector fields on $D_p$  with adapted $\C^*$-fibration can be found in the work of Leuenberger \cite{Leuenberger} who up to automorphism classifies all complete algebraic vector fields 
on Danielewski surfaces. Interesting examples (whose  flow maps are not algebraic) are fields whose adapted $\C^*$-fibration is given by  $(x,y,z) \mapsto x^m(x^l (z +a)+Q(x))^n$ for coprime numbers $m, n \in \N$, 
$a\in \C$ and $0 \leq l < \deg(Q)$. The exact formula for these fields can be found in the Main Theorem of loc.cit.
\end{remark}

\begin{remark} 
Without specifying a concrete automorphism which is  not in the group generated by overshears, Anders\'en and Lempert use  an abstract Baire category argument in~\cite{AndersenLempert} to show  that the group generated by overshears 
in $\C^n$ is a proper subroup of the group of holomorphic automorphisms $ \rm{Aut}_{hol} (\C^n)$ of $\C^n$.   We do believe that such a proof could work in the case of Danielewski surfaces as well.
\end{remark}

\def\listing #1#2#3{{\sc #1}:\ {\it #2},\ #3.}


\begin{thebibliography}{9999}

\bibitem[Akh]{Akhiezer}\listing{ D. N. Akhiezer}{ Lie group actions in complex analysis}{ Aspects of Mathematics, E27. Friedr. Vieweg  Sohn, Braunschweig, 1995. viii+201 pp.}
\bibitem[An1]{Andersen}\listing{E. Anders\'en}{Volume-preserving automorphisms of $\C^n$}{Complex Variables Theory Appl. \textbf{14} (1990), no. 1-4, 223--235}
\bibitem[An2]{Andersen2}\listing{E. Anders\'en}{Algebraic and analytic properties of groups of holomorphic automorphisms of $\C^n$}{Doctorial Thesis, Lunds tekniska h\"ogskola, ISSN: 0347-8475, 1994}
\bibitem[AL]{AndersenLempert}\listing{E. Anders\'en, L. Lempert}{On the group of holomorphic automorphisms of $\C^n$}{Invent. Math.  \textbf{110} (1992),  no. 2, 371--388}

\bibitem[AF]{AhernForstneric}\listing{P. Ahern, F. Forstneri$\check{\text{c}}$}{One parameter automorphism groups on $\C^2$}{ Complex Variables Theory Appl.  \textbf{27} (1995),  no. 3, 245--268}
\bibitem[AFV]{AhernForstnericVarolin}\listing{P. Ahern, F. Forstneri$\check{\text{c}}$, D. Varolin}{Flows on $\C^2$ with polynomial time one map}{Complex Variables Theory Appl. \textbf{29} (1996),  no. 4, 363--366}
\bibitem[AR]{AhernRudin}\listing{P. Ahern, W. Rudin}{Periodic automorphisms of $\C^n$}{Indiana Univ. Math. J. \textbf{44} (1995),  no. 1, 287--303}

\bibitem[AKL]{AndristKutzschebauchLind}\listing{R. Andrist, F. Kutzschebauch, and A. Lind}{Holomorphic Automorphisms of Danielewski surfaces II -- structure of the overshear group}{Manuscript}


\bibitem[Cr]{Crachiola}\listing{A. J. Crachiola}{On automorphisms of Danielewski surfaces}{J. Algebraic Geom. \textbf{15} (2006), no. 1, 111--132}
\bibitem[Da1]{Daigle}\listing{D. Daigle}{Locally nilpotent derivations and Danielewski surafaces}{Osaka J. Math. \textbf{41} (2004), no. 1, 37--80}
\bibitem[Da2]{Daigle2}\listing{D. Daigle}{On locally nilpotent derivations of $k[X\sb 1,X\sb 2,Y]/(\phi(Y)-X\sb 1X\sb 2)$}{J. Pure Appl. Algebra \textbf{181} (2003), no. 2-3, 181--208}

\bibitem[Du1]{Dubouloz}\listing{A. Dubouloz}{Additive group actions on Danielewski varieties and the cancellation problem}{Math. Z. \textbf{255} (2007), no. 1, 77--93}
\bibitem[DP]{DuboulozPoloni}\listing{A. Dubouloz, P.M. Poloni}{On a class of Danielewski surfaces in affine 3-space}{J. Algebra \textbf{321} (2009), no. 7, 1797--1812}

\bibitem[Dy]{Dynkin}\listing{Dynkin, E. B.}
{Calculation of the coefficients in the Campbell-Hausdorff formula. (Russian)} 
{Doklady Akad. Nauk SSSR (N.S.) 57, (1947). 323--326} 
\bibitem[Fa]{Fabritiis3}\listing{C. De Fabritiis}{One-parameter groups of volume-preserving automorphisms of $\C^2$}{Rend. Istit. Mat. Univ. Trieste \textbf{26} (1994), no. 1-2, 21--47}
\bibitem[FoK]{ForstnericKutzschebauch} \listing{{F. Forstneri$\check{\text{c}}$}, F. Kutzschebauch}   
{The first thirty years of Andersen-Lempert theory}
{arXiv:2111.08802}


\bibitem[KK]{KutzschebauchKaliman}\listing{S. Kaliman, F. Kutzschebauch}{Density property for hypersurfaces $UV=P(\overline{X})$}{Math. Z. \textbf{258} (2008), no. 1, 115--131}

\bibitem[KKL]{KalimanKutzschebauchLeuenberger} \listing{S.Kaliman,  F. Kutzschebauch,M. Leuenberger}{Complete algebraic vector fields on affine surfaces}{ Internat. J. Math. 31 (2020), no. 3, 2050018, 50 pp.} 
\bibitem[Iw]{Iwasawa}Iwasawa, Kenkichi (1949). "On Some Types of Topological Groups". Annals of Mathematics. 50 (3): 507--558
\bibitem[Ku]{Kutzschebauch} \listing{F. Kutzschebauch}{  Compact and reductive subgroups of the group of holomorphic automorphisms of $\C^n$}{Singularities and complex analytic geometry (Japanese) (Kyoto, 1997). Surikaisekikenkyusho Kokyuroku No. 1033 (1998), 81--93} 
\bibitem[KKr]{KutzschebauchKraft}\listing{F. Kutzschebauch, H.P. Kraft}{Equivariant Affine Line Bundles and Linearization}{Math. Res. Lett. \textbf{3} (1996), no. 5, 619--627}
\bibitem[KL]{KutzschebauchLind}\listing{F. Kutzschebauch, A. Lind}{Holomorphic Automorphisms of Danielewski surfaces I -- density of the group of overshears}{ Proc. Amer. Math. Soc. 139 (2011), no. 11, 3915--3927}


\bibitem[Leu]{Leuenberger} \listing{M. Leuenberger} {Complete algebraic vector fields on Danielewski surfaces} {Ann. Inst. Fourier (Grenoble) \textbf{66} (2016), no. 2, 433--454}

\bibitem[ML1]{MakarLimanov}\listing{L. Makar-Limanov}{On groups of automorphisms of a class of surfaces}{Israel J. Math. \textbf{69} (1990), no. 2, 250--256}
\bibitem[ML2]{MakarLimanov2}\listing{L. Makar-Limanov}{On the group of automorphisms of a surface $x\sp ny=P(z)$}{Israel J. Math. \textbf{121} (2001), 113--123}

\bibitem[Le]{Levi} E. E. Levi: "Sulla struttura dei gruppi finiti e continui", Atti della Reale Accademia delle Scienze di Torino, Band XL: Seiten 551--565

\bibitem[Ma]{Malcev}\listing{A.I. Malcev}{On a class of homogeneous spaces. (Russian)}{Izvestiya Akad. Nauk. SSSR. Ser. Mat. \textbf{13},  (1949). 9--32}



\bibitem[Sr]{Serre}\listing{Serre J.P.}{Trees}{Springer-Verlag, Berlin-New York, 1980}

\bibitem[SL]{SilvaLeite} \listing{F. Silva Leite}{Bounds on the order of generation of SO(n,R) by one-parameter subgroups}{ Current directions in nonlinear partial differential equations (Provo, UT, 1987). Rocky Mountain J. Math. 21 (1991), no. 2, 879--911} 



\end{thebibliography}
\end{document}